\documentclass{amsart}

\usepackage{amssymb,amsmath,amsthm,amscd}

\usepackage{color}

\usepackage{enumerate}

\newcommand{\Z}{\mathbb Z}

\newcommand{\N}{\mathbb N}

\numberwithin{equation}{section}
\newtheorem{theorem}[equation]{Theorem}
\newtheorem{corollary}[equation]{Corollary}

\theoremstyle{definition}
\newtheorem{definition}[equation]{Definition}

\newtheorem{question}[equation]{Question}

\theoremstyle{remark}

\begin{document}

\title[Symbolic and Borel work of M. Hochman]{The work of Mike Hochman  on multidimensional symbolic dynamics  and Borel dynamics}


\author[M. Boyle]{Mike Boyle}
\thanks{This
  article is based on the talk I gave on the occasion of Mike
Hochman's being awarded the  Brin Prize. 
I am grateful to Nishant Chandgotia, Emmanuel Jeandel, Doug Lind and
Anush Tserunyan 
for very helpful feedback.}

\address{
Department of Mathematics, University of Maryland, College
Park, MD 20742-4015, USA
}
\email{mmb@math.umd.edu}





\begin{abstract}
  We review the impact of Mike Hochman's work on
mutlidimenisonal
symbolic dynamics and Borel dynamics.
\end{abstract}

\maketitle





\section{Introduction}



In an ancient story, several blind men encounter an elephant.
One feels the ears, one the trunk, one the tusks, and so on. 
All features are dramatic and important. 
The blind men have quite different opinions  about what makes an elephant.

Mike Hochman's research is like the elephant -- 
with several dramatic features, quite varied; 
including important work on 
  fractals, 
 multidimensional symbolic dynamics,
  Borel dynamics,  
  ergodic theory and more.
  In this short note, I only look quickly at parts of two features of
  Hochman's elephant -- his 
  most important work (so far) in 
  multidimensional symbolic dynamics\begin{footnote}{There is 
    other  excellent work on multidimensional symbolic dynamics.}\end{footnote}
and Borel dynamics. 

%

\section{Multidimensional symbolic dynamics} 
I'll focus on the content and impact of the papers  
   \cite{HM} (joint with  Tom Meyerovitch) 
  and \cite{H1}. 

  \subsection{Definitions}

  ``Multidimensional symbolic dynamics''
refers here to $\Z^d$ subshifts, $(X,\sigma )$ ,
$d=2,3,...$ .
Here $X$ is a closed, shift invariant subset of 
$A^{\Z^d}$, for some finite alphabet $A$, with  the shift action by
$\Z^d$: for $v\in \Z^d$, $(\sigma^v x)(n) = x(n+v)$.
$A^{\Z^d}$ has the product topology of the discrete topology on $A$, and 
$X$ has the relative topology.  
An element $x$ of $X$ is a function from $\Z^d$ into some finite alphabet $A$;  
$x$ is visualized as a way of filling the $\Z^d$ lattice with symbols from $A$.

Let $C_n(x)$ be the restriction of $x$ to
the cube $\{ 0, 1, ... , n-1\}^d$.
The $\Z^d$ topological entropy of $(X, \sigma )$ is
\[
\lim_n (1/n^d) \log | \{ C_n (x) : x \in X \}| \ .
\]
$X$ is a $\Z^d$ shift of finite type (SFT) if there is some positive integer
$n$ and some finite set $F$ such that $X$ is the subset of 
$A^{\Z^d}$ consisting of the $x$ such that for all $v$ in $\Z^d$,
$C_n(\sigma^vx) \notin F$.  When $X= A^{\Z^d}$, the SFT is the
full $\Z^d$ shift on $|A|$ symbols. 

A $d$ dimensional cellular automaton is a continuous
shift-commuting map from a full $\Z^d$ shift
into itself.

\subsection{Before Hochman and Meyerovitch.}

Let's sketch the main features of the state of knowledge of the
possible entropies of $\Z^d$ SFTs before the Hochman-Meyerovitch paper \cite{HM}:   

\begin{enumerate}
\item 
  ($d=1$) 
  The set of possible topological entropies of a $\Z$ SFT
  was known to be the set of
  logarithms  of an easily understood class of algebraic integers
  \cite{LindPerron1984}.

\item 
  For $d>1$,  a
  limited  collection of  $\Z^d$ SFT entropies had been computed exactly,   
  by (highly nontrivial) work of mathematical physicists (e.g.
  \cite{Baxter,Kasteleyn1961,Lieb}). 

%
  
\item For $d>1$,
  it was well known that there can be no algorithm which
  takes a finite presentation of 
   a general 
$\Z^d$ SFT or $d$-dimensional cellular automata
  and 
  computes
  its entropy (or, roughly speaking, any nontrivial property--see
      e.g. \cite{KariRice94,Kari16}).
 
\end{enumerate}


%
%
%

\subsection{The main theorems} 
     
\begin{theorem} \label{HM} \cite{HM}  
Suppose $h$ is a nonnegative real number 
and $d$ is an integer, $d\geq 2$.
Then the following are equivalent. 
\begin{enumerate}
\item $h$ is the topological entropy of a $\Z^d$ shift of finite type.

\item There is a Turing machine which produces a
 sequence of nonnegative numbers
 $h_n$ such that $h=\inf h_n$.
\end{enumerate} 
\end{theorem} 
The countable class of numbers in (2) 
properly contains the nonnegative real numbers $h$
which can be algorithmically approximated to arbitrary precision. 
For example:  
 $\pi, e, \gamma $,   algebraic numbers ... 
But a general procedure for estimating numbers  in this class 
can  only only produce converging upper bounds -- no general procedure
will give  an error bound for the estimate.

Hochman's paper \cite{H1} showed Theorem \ref{HM}
to be part of a systematic approach. Here is an example.

\begin{theorem} \label{H1a}  \cite{H1} 
 Suppose $h$ is a nonnegative real number  
 and $d\geq 3$ is an integer\begin{footnote}{This theorem was later extended
     to  $d=2$ in the papers \cite{AubrunSablik, DRS2010}, and then to
      $d=1$ in \cite{GuillonZinoviadis}.}\end{footnote}.
The following are equivalent. 

\begin{enumerate} 
\item $h$ is the topological entropy of a $\Z^d$ cellular automaton. 

\item  There is a Turing machine which produces a
 sequence of nonnegative numbers
$h_n$ such that $h=\liminf h_n$. 

\end{enumerate}
\end{theorem} 

For $h= \inf h_n$, from the sequence $(h_n)$
you at least get at any finite input an upper bound to $h$. 
For $h= \liminf h_n$, from any finite set of terms  from
the sequence 
$(h_n)$ you obtain no bound at all on $h$.

For Theorems \ref{HM} and \ref{H1a}, 
the proof that (1) implies  (2) is not hard.
So, there were two difficulties.  Initially, 
to have the audacity to imagine the
theorems. Then, 
to  produce actual constructions 
showing  $(2) \implies (1)$. 
A starting point for these constructions 
was the  1989  paper \cite{Mozes} of Shahar Mozes, 
which used ideas from
 Raphael Robinson's  
1971 paper \cite{Robinson71} to construct  
 $ \Z^2$ SFT presentations for many planar substitution tilings. 

Theorem \ref{H1a} is dramatic and easily stated, but 
the paper \cite{H1} fundamentally is about simulating
effective subshifts with $\Z^d$ shifts of finite type ($d\geq 3$). 
For example, in \cite{H1} this   gives (for $d\geq 3 $) a 
recursion theoretic characterization of the possible  dynamics
of a $d$-dimensional cellular automaton on its limit set  
(up to a modest equivalence relation). Simulation 
results for $d\geq 3$ were generalized to $d=2$ (optimal)
in the papers \cite{AubrunSablik, DRS2010}. 

\subsection{Impact of the papers} 

Before \cite{H1,HM},
workers were very aware of 
the \lq\lq swamp of undecidability\rq\rq\begin{footnote}{This 
  section title from  Lind's 2004 survey
  \cite{LindSurvey2004} reflected the zeitgeist.}\end{footnote} 
   as an obstacle to proving something (perhaps anything)
general and   interesting  
about multidimensional cellular automata and SFTs. 
The papers \cite{H1,HM}  did \lq\lq mathematical judo\rq\rq\ on recursion theory: 
making it a friend instead of an enemy. It seems fair to call this a paradigm shift. 
The papers offered   a blueprint for characterizing the
range of invariants of $\Z^d$ SFTs (or  sofic shifts, or
effective shifts):
\begin{enumerate}
  \item 
 Find the \lq\lq obvious\rq\rq\ recursion theoretic obstruction. 
\item  Make constructions to prove there is no other obstruction.
  \end{enumerate} 
Work  by various people 
has since been carried out in this vein, 
to characterize for example the possible 
entropy dimensions of $\Z^d$ SFTs, for $d\geq 3$
\cite{Meyerovitch2010};
periods of of $\Z^d$ SFTs, for $d\geq 2$
\cite{JeandelVanier2015}; 
sets of expansive directions for $\Z^2$ SFTs
\cite{Z2016}; 
sets of limit measures of cellular automata 
\cite{delacourtdemenibus, deMenibusSablik}; 
and so on. 
%
%
%
%
 It now seems that to a large extent
 the   landscape of possibilities for 
 {\it general}  
multidimensional SFTs (or sofic shifts, or effective shifts) 
has a  recursion theoretic description.\begin{footnote}{I 
emphasize {\it general}. Our understanding  of 
constraints in special classes remains poor;
additional constraints do arise
(see e.g.  \cite{PavlovSchraudner2015}). Moreover, 
 McGoff and Pavlov \cite{McgoffPavlov16}
 suggest, and prove cases of, an interesting heuristic:
 for $d\geq 2$, 
where general $\Z^d$ SFTs may be badly behaved, 
in a precise sense 
\lq\lq typical\rq\rq\ $\Z^d$ SFTs can be  well
behaved.}\end{footnote}
\lq\lq Effective subshifts\rq\rq\  
have also been considered
for more general groups 
\cite{AubrunBarbieriSablik}.  
For more, see the surveys \cite{Hsurvey,Jeandel}
of Hochman and  Jeandel.

The papers  \cite{H1,HM}  
also furthered 
the rich interdisciplinarity of the
multidimensional symbolic dyamics area, 
as it spurred contributions from 
logicians and experts in
recursive tiling constructions  
(e.g. \cite{DurandRomaschenko2017,DRS2010,Simpson2014}).

  \section{Borel dynamics} 

\subsection{Definitions} 

  A standard Borel space is a pair $(X,\Sigma )$ such that
  $X$ is a set, and $\Sigma$ is the $\sigma$-algebra (the Borel
  $\sigma$-algebra) generated
  by the open sets defined by a given complete, separable metric on $X$.
  A Borel set is an element of $\Sigma$. 
  We usually suppress the $\sigma$-algebra from the notation. 
  A morphism in the Borel category
  is a map for which the inverse image
  of every Borel set is a Borel set.
  Standard Borel spaces of equal cardinality are isomorphic.

In this article, a Borel system $(X,T)$ is a Borel automorphism $T\colon X\to X$
 of a standard  Borel space. (So, we restrict to $\Z$ actions.) 
  We say  a Borel system $(X,T)$ is free if the $\Z$ action  generated by $T$ 
  is free, i.e., $T$ has no periodic point.

 A {\it null
  set} for a Borel system  $(X,T)$  is a set which has
 measure zero for every $T$-invariant Borel probability. 
 A {\it full set} is the complement of a null set.
 $\mathcal M(T)$ is the set of $T$-invariant Borel probabilities.

 For a Borel system $(X,T)$, pick a countable collection of
 sets generating the Borel $\sigma$-algebra; if 
 $(X,T)$ comes with a topology, the countable collection is
 chosen to be a basis for the topology. We say a point in $X$ is
 {\it generic} for a measure $\mu$ in $ \mathcal M(T)$ if under $T$ and $T^{-1}$ it 
 it visits all sets in the countable collection with ergodic-theorem
 frequencies. Let $X_g$ be the union over $\mu$ in $\mathcal M(T)$
 of the $\mu$-generic points. Let $X_p$ be the set of $T$-periodic points,
 a subset of $X_g$. Set $X_d = X\setminus X_g$. 

 A Borel system $(X,T)$ is the disjoint union of subsystems given
 by restriction to $X_p$, $X_g \setminus X_p$ and $X_d$. 
 Then $X_g$ is a $T$-invariant full set, and
$X_g \setminus X_p$ supports all nonatomic measures in $\mathcal M(T)$. 
 A complete Borel invariant
 for $(X_p,T)$ is simply the function giving the cardinality of
 the set of orbits of size $n$, for $n\in \N$. So, understanding
 of the Borel system amounts to understanding the free systems
 $(X_g\setminus X_p,T)$ and $(X_d,T)$. By ergodic decomposition,
 the system $(X_d,T)$ supports no $T$-invariant Borel probability.

\subsection{Borel dynamics off null sets, before Hochman}     

  It is fundamental in dynamics to consider a homeomorphism $T\colon X\to X$
  with respect to some  $T$-invariant Borel probability $\mu$,
  neglecting the sets of $\mu$ measure zero. More ambitiously,
  one could try to understand $T$ simultaneously with respect to
  a large subset of $\mathcal M(T)$.  
  In this spirit, motivated by his study of low dimensional piecewise smooth dynamics,
  Buzzi proposed the following definition.\begin{footnote}
  {In a  similar spirit,
    Rufus Bowen in his 1973 classic
    \cite{Bowen1973} introduced a closely related notion, which
    he also called \lq\lq entropy conjugacy\rq\rq . However, 
    a map implementing Bowen's  entropy conjugacy is required
    to be continuous, not only Borel, on a suitably large set.}
  \end{footnote}

\begin{definition} \cite{Buzzi97}  Suppose $ S,T$ are Borel systems with  
 equal and finite topological entropy $h$.
Then $S$ and $T$  are {\it entropy conjugate} if for some $\epsilon >0$
they are isomorphic
modulo sets of measure zero for all ergodic invariant Borel probabilities
of entropy $\geq h-\epsilon $.
\end{definition}

Buzzi showed certain piecewise smooth systems of equal entropy are entropy
conjugate to (countable state) Markov shifts,
and asked whether recurrent Markov shifts of
equal entropy are entropy conjugate. 
Mixing shifts of finite type of equal entropy are entropy conjugate,
as a consequence of the Adler-Marcus Theorem \cite{AdlerMarcus1979}.
By different explicit codings, mixing strongly positive recurrent (SPR) 
Markov shifts\begin{footnote}{The
 SPR shifts are the natural subclass of positive
  recurrent Markov shifts whose properties most
  closely resemble those of finite state irreducible Markov shifts. 
They are the positive recurrent Markov shifts for which the measure of
maximal entropy is exponentially
recurrent.   
  For  characterizations of this
  class in terms of subsystems
  and local zeta functions,
  see \cite[Proposition 2.3 and the definition (4) above it]{BBG2006}.}\end{footnote}
of equal finite entropy were shown to be entropy conjugate
\cite{BBG2006}. 
Entropy conjugacy was not
understood even for general positive recurrent Markov shifts. 

  \subsection{Borel dynamics on full sets, after Hochman} 

%
%
%

To understand the impact of  Hochman's paper \cite{H2} on this topic, 
we begin with some  definitions.  
 For $t>0$, let $\mathcal B_t$ be the collection of Borel systems
which are free  and have no invariant
Borel probability of entropy $\geq t$.
A Borel system $(X,T)$ is strictly $t$-universal if
it is in $\mathcal B_t$ and  every
member of $\mathcal B_t$ embeds
modulo null sets to a subsystem of $(X,T)$. 
The entropy of a Borel system $(X,T)$
  is defined to be the supremum of $h_{\mu}(T)$, over $T$-invariant Borel probabilities $\mu$.

  For a Borel system $(X,T)$, 
form $X'$ by removing periodic points and points generic for any
ergodic measure of maximal entropy.

\begin{theorem} \label{theoremH2}
      \cite{H2}  Suppose a Borel system $(X,T)$ of finite
       entropy $t>0$ contains
mixing shifts of finite type with entropies arbitrarily close to $t$. 
\begin{enumerate}
\item Then $(X',T)$ is strictly $t$-universal.

\item Strictly $t$-universal systems are isomorphic mod null sets.
\end{enumerate} 
\end{theorem}

\begin{proof} We give a brief indication of how the proof goes.

(2) Not hard. (Cantor-Bernstein argument of set theory works in this category.)

(1) As follows:

{\it   (I) Broad strategy.}
B.   Weiss showed a free Borel system has a countable generator,
hence embeds to a subsystem of $((1,2,... )^{\Z}, \sigma)$, the
full shift over a countable alphabet.
  Then, it's not hard to show, 
  given $t>t-\epsilon >0$ and a mixing
  SFT of entropy $t$,  
  that 
  it suffices to find 
 $B$  a Borel subsystem of
$((1,2,... )^{\Z}, \sigma)$  supporting all 
  the ergodic measures of entropy $\leq t-\epsilon$,
  and Borel embed $B$ into that SFT. 

{\it (II) Finer strategy.} 
\lq\lq Observe\rq\rq:  the argument of the Krieger generator theorem
gives a Borel map which on the generic points of 
every ergodic Borel measure, individually, 
is injective with finitary inverse. 

{\it (III) Hard construction step.}  Augment  the coding  of that map to
make it injective on the union of all those supports. 

\end{proof}

  \begin{corollary}\label{corollarylist}
    \cite{H2} For every  $(X,T)$ of equal positive entropy $t$
  from the following collection, 
$(X',T)$ is the same strictly $t$-universal system: 
mixing SFTs, mixing countable state Markov shifts, mixing finitely presented
systems, Anosov diffeomorphisms ...
\end{corollary} 

  Corollary \ref{corollarylist}
  leaves open the nature of measures of maximal entropy for
  the listed systems. For example, the mixing Markov shifts which are not
  positive recurrent have no measure of maximal entropy. But, the 
  other systems listed have   a unique measure of maximal entropy,
which is Bernoulli. This gives the following. 

\begin{corollary}\label{list2} \cite{H2}
    Entropy is a complete
  invariant of Borel isomorphism modulo periodic points and null sets
    among  systems in the following collection:
  mixing SFTs, mixing positive recurrent Markov shifts, mixing finitely presented
systems, Anosov diffeomorphisms ...
\end{corollary} 

We see that Hochman not only resolved Buzzi's entropy conjugacy question, but
with a more insightful viewpoint proved something much more general and
fundamental. 

There have been further developments in this vein. 
Recall that Hochman's argument begins with
the Krieger generator theorem \cite{Krieger70}, using mixing SFTs.
Quas and Soo
\cite{QuasSoo16} proved a
 generator theorem for a much more general class, the homeomorphisms
of compact metric spaces satisfying (i) almost weak specification, (ii) the
small boundaries property and (iii) asymptotic h-expansiveness (or more generally
upper semicontinuity of the entropy map on invariant Borel probabilities). 
The Quas-Soo theorem was striking both in its generality and in its liberation
from shifts of finite type -- it applies to many systems, such as quasihyperbolic
toral automorphisms, which contain no infinite shift of finite type.

B. Weiss (unpublished) showed the requirement of (iii) in the Quas-Soo Theorem
could be dropped. Recently,
Chandgotia and Meyerovitch (to be posted) removed the requirement (ii) and,
dramatically, prove the universality result of Hochman, but with
mixing  systems with almost weak specification in place of mixing
SFTs.\begin{footnote}{{\it Note added in proof:}
This result is also contained in the work
\cite{burguet} of David Burguet, appearing in the 2019
arxiv post \cite{burguetarxiv}. The work
\cite{ChanMey} of Chandgotia and Meyerovitch was
presented
at an August 2018 PIMS workshop; the Burguet paper
\cite{burguet}  wasf
submitted to ETDS in August 2017. Neither of the
independent works \cite{burguet,ChanMey} contains the
other.}\end{footnote}
  Chandgotia and Meyerovitch also generalize the result to $\Z^d$ actions. 
  Altogether, these works comprise a remarkable advance in our understanding of
  how positive entropy in a class of systems can force dynamical complexity.

Away from mixing systems, Hochman's universality
approach to isomorphism on full sets
can become more complicated, but remain useful (see e.g. \cite{BB}).

%

\subsection{Borel dynamics beyond probabiities}

  For a Borel system $(X,T)$, 
  we have seen the success of Hochman's universality approach
  to the action of $T$ on  $X_g$. What about $X_d$?
  Although $(X_d,T)$ supports no invariant Borel probability,
  it may still admit complicated dynamics (for example,
  $\sigma$-finite or nonsingular invariant measures). 
Commonly, 
$X_d$ will be a dense $G_{\delta}$ in a topological space $X$.  
One might think of $X_d$ as the  \lq\lq dark matter\rq\rq : 
 $X_d$ is often large,  and the dynamics of $(X_d,T)$ 
can be hard to see.  

Let $\mathcal D$ be the class of Borel
systems admitting no invariant Borel probability.
The study of the class $\mathcal D$  began 
with Shelah and Weiss
\cite{Shelahweiss,Weiss84,Weiss89}, and then Kechris and
others (see \cite{DoughertyJacksonKechris1994,JacksonKechrisLouveau2002}). 
  Earlier, Krengel had shown that for any infinite  measure preserving 
  ergodic  conservative system, there is a two-set generator. (This
  is a.e., neglecting measure zero sets.) Weiss's 1989 paper
  \cite{Weiss89} included the following question
  (this question was later asked for a general countable group in
  \cite[Problem 5.7]{JacksonKechrisLouveau2002}).

\begin{question} \cite{Weiss89} Suppose a Borel system admits no
 invariant Borel probability. Must it have a finite generator?
  A 2-set generator? 
\end{question}   

This was a rather fundamental, longstanding question.
The title of Hochman's paper 
 \cite{H3} gives the answer: 
\lq\lq Every Borel automorphism without finite invariant
  measure admits a two-set generator.\rq\rq

  We  note here the related theorem of Tserunyan:

  \begin{theorem} 
 \cite{Tserunyan}
  Suppose $G$ is an arbitrary countable group, acting by homeomorphisms
  on a $\sigma$-compact Polish space, which does not admit an invariant Borel
  probability. Then $G$ has a finite generator.
  (In fact, a 32-generator.)
\end{theorem}

  Tserunyan's  theorem is remarkable for the generality of the acting group.
  However, not every $\Z$-action on a Borel space is Borel conjugate
  to a continuous action on a $\sigma$-compact space
  \cite{ConleyKechrisMiller2013}.
  The two theorems seem quite different.
  Certainly, the proofs are very different. 

  Hochman's proof of the 2-generator theorem is much harder than
  the proof of Theorem \ref{theoremH2}. 
  \cite{H2}. A big problem is to even find a strategy. We refer to
\cite{H3} for a clear explanation, and a list of compelling open problems.  
The paper also contains the following  striking result.   

\begin{theorem} \label{theoremboreluniv} \cite{H3}
%
%
  If a Borel system
  $(X,T)$ contains an infinite 
  mixing SFT, then $(X_d,T)$ is the unique
  (up to Borel conjugacy) Borel system in $\mathcal D$ 
  into which every system in $\mathcal D$ embeds.
\end{theorem}
   
So, for many familiar systems,
   (e.g., those on the list of
Corollary \ref{corollarylist}),
 $(X_d,T)$ is the
  same universal \lq\lq dark matter\rq\rq\  Borel dynamics. 
  (Whatever that is ... )

  \begin{corollary} \cite{H3} Suppose two homeomorphisms 
 of equal finite positive entropy 
  are mixing and lie in any of  following classes:
  SFT, sofic shift,   positive recurrent
  countable state Markov shift,
  finitely presented system (e.g. Anosov homeomorphism,  or
  Axiom A on a basic set).

  Then their restrictions to the complement of the periodic
  points  are Borel isomorphic. 
\end{corollary} 

There is now  an obvious question.
  \begin{question} 
    Suppose $(X,T)$ is a mixing homeomorphism of an
    infinite compact metric
  space satisfying almost weak specification. 
  Must     $(X_d,T)$ be the universal \lq\lq dark matter\rq\rq\ dynamics?
  \end{question}
  
 \bibliographystyle{plain}
\bibliography{hb}

\begin{thebibliography}{10}

\bibitem{AdlerMarcus1979}
Roy~L. Adler and Brian Marcus.
\newblock Topological entropy and equivalence of dynamical systems.
\newblock {\em Mem. Amer. Math. Soc.}, 20(219):iv+84, 1979.

\bibitem{AubrunBarbieriSablik}
Nathalie Aubrun, Sebasti\'an Barbieri, and Mathieu Sablik.
\newblock A notion of effectiveness for subshifts on finitely generated groups.
\newblock {\em Theoret. Comput. Sci.}, 661:35--55, 2017.

\bibitem{AubrunSablik}
Nathalie Aubrun and Mathieu Sablik.
\newblock Simulation of effective subshifts by two-dimensional subshifts of
  finite type.
\newblock {\em Acta Appl. Math.}, 126:35--63, 2013.

\bibitem{Baxter}
Rodney~J. Baxter.
\newblock {\em Exactly solved models in statistical mechanics}.
\newblock Academic Press, Inc. [Harcourt Brace Jovanovich, Publishers], London,
  1989.
\newblock Reprint of the 1982 original.

\bibitem{Bowen1973}
Rufus Bowen.
\newblock Topological entropy for noncompact sets.
\newblock {\em Trans. Amer. Math. Soc.}, 184:125--136, 1973.

\bibitem{BB}
Mike Boyle and J\'er\^ome Buzzi.
\newblock The almost {B}orel structure of surface diffeomorphisms, {M}arkov
  shifts and their factors.
\newblock {\em J. Eur. Math. Soc. (JEMS)}, 19(9):2739--2782, 2017.

\bibitem{BBG2006}
Mike Boyle, Jerome Buzzi, and Ricardo G\'omez.
\newblock Almost isomorphism for countable state {M}arkov shifts.
\newblock {\em J. Reine Angew. Math.}, 592:23--47, 2006.

\bibitem{burguet}
D.~Burguet.
\newblock Topological and almost {B}orel universality for systems with the weak
  specification property.
\newblock {\em Ergodic Theory Dynam. Systems}, (first published online
  Februrary 2019).

\bibitem{burguetarxiv}
David Burguet.
\newblock Topological and almost {B}orel universality for systems with the weak
  specification property.
\newblock {\em preprint}, arXiv:1901.00666, 2019.

\bibitem{Buzzi97}
J\'er\^ome Buzzi.
\newblock Intrinsic ergodicity of smooth interval maps.
\newblock {\em Israel J. Math.}, 100:125--161, 1997.

\bibitem{ChanMey}
Nishant Chandgotia and Tom Meyerovitch.
\newblock Borel subsystems and ergodic universality for compact
  $\text{{Z}}^d$-systems via specification and beyond.
\newblock {\em preprint}, arXiv:1903.05716, 2019.

\bibitem{ConleyKechrisMiller2013}
Clinton~T. Conley, Alexander~S. Kechris, and Benjamin~D. Miller.
\newblock Stationary probability measures and topological realizations.
\newblock {\em Israel J. Math.}, 198(1):333--345, 2013.

\bibitem{delacourtdemenibus}
Martin Delacourt and Benjamin Hellouin~de Menibus.
\newblock Characterisation of limit measures of higher-dimensional cellular
  automata.
\newblock {\em Theory Comput. Syst.}, 61(4):1178--1213, 2017.

\bibitem{DoughertyJacksonKechris1994}
R.~Dougherty, S.~Jackson, and A.~S. Kechris.
\newblock The structure of hyperfinite {B}orel equivalence relations.
\newblock {\em Trans. Amer. Math. Soc.}, 341(1):193--225, 1994.

\bibitem{DurandRomaschenko2017}
Bruno Durand and Andrei Romashchenko.
\newblock On the expressive power of quasiperiodic {SFT}.
\newblock In {\em 42nd {I}nternational {S}ymposium on {M}athematical
  {F}oundations of {C}omputer {S}cience}, volume~83 of {\em LIPIcs. Leibniz
  Int. Proc. Inform.}, pages Art. No. 5, 14. Schloss Dagstuhl. Leibniz-Zent.
  Inform., Wadern, 2017.

\bibitem{DRS2010}
Bruno Durand, Andrei Romashchenko, and Alexander Shen.
\newblock Effective closed subshifts in 1{D} can be implemented in 2{D}.
\newblock In {\em Fields of logic and computation}, volume 6300 of {\em Lecture
  Notes in Comput. Sci.}, pages 208--226. Springer, Berlin, 2010.

\bibitem{GuillonZinoviadis}
Pierre Guillon and Charalampos Zinoviadis.
\newblock Densities and entropies in cellular automata.
\newblock In {\em How the world computes}, volume 7318 of {\em Lecture Notes in
  Comput. Sci.}, pages 253--263. Springer, Heidelberg, 2012.

\bibitem{deMenibusSablik}
Benjamin Hellouin~de Menibus and Mathieu Sablik.
\newblock Characterization of sets of limit measures of a cellular automaton
  iterated on a random configuration.
\newblock {\em Ergodic Theory Dynam. Systems}, 38(2):601--650, 2018.

\bibitem{H1}
Michael Hochman.
\newblock On the dynamics and recursive properties of multidimensional symbolic
  systems.
\newblock {\em Invent. Math.}, 176(1):131--167, 2009.

\bibitem{H2}
Michael Hochman.
\newblock Isomorphism and embedding of {B}orel systems on full sets.
\newblock {\em Acta Appl. Math.}, 126:187--201, 2013 (arXiv 2010).

\bibitem{H3}
Michael Hochman.
\newblock Every {B}orel automorphism without finite invariant measures admits a
  two-set generator.
\newblock {\em Math arXiv; JEMS, to appear}, 1508.02335, 2015.

\bibitem{Hsurvey}
Michael Hochman.
\newblock Multidimensional shifts of finite type and sofic shifts.
\newblock In {\em Combinatorics, words and symbolic dynamics}, volume 159 of
  {\em Encyclopedia Math. Appl.}, pages 296--358. Cambridge Univ. Press,
  Cambridge, 2016.

\bibitem{HM}
Michael Hochman and Tom Meyerovitch.
\newblock A characterization of the entropies of multidimensional shifts of
  finite type.
\newblock {\em Ann. of Math. (2)}, 171(3):2011--2038, 2010.

\bibitem{JacksonKechrisLouveau2002}
S.~Jackson, A.~S. Kechris, and A.~Louveau.
\newblock Countable {B}orel equivalence relations.
\newblock {\em J. Math. Log.}, 2(1):1--80, 2002.

\bibitem{Jeandel}
Emmanuel Jeandel.
\newblock Computability in symbolic dynamics.
\newblock In {\em Pursuit of the universal}, volume 9709 of {\em Lecture Notes
  in Comput. Sci.}, pages 124--131. Springer, [Cham], 2016.

\bibitem{JeandelVanier2015}
Emmanuel Jeandel and Pascal Vanier.
\newblock Characterizations of periods of multi-dimensional shifts.
\newblock {\em Ergodic Theory Dynam. Systems}, 35(2):431--460, 2015.

\bibitem{KariRice94}
Jarkko Kari.
\newblock Rice's theorem for the limit sets of cellular automata.
\newblock {\em Theoret. Comput. Sci.}, 127(2):229--254, 1994.

\bibitem{Kari16}
Jarkko Kari.
\newblock Cellular automata, tilings and (un)computability.
\newblock In {\em Combinatorics, words and symbolic dynamics}, volume 159 of
  {\em Encyclopedia Math. Appl.}, pages 241--295. Cambridge Univ. Press,
  Cambridge, 2016.

\bibitem{Kasteleyn1961}
P.W. Kasteleyn.
\newblock The statistics of dimers on a lattice.i.
\newblock {\em Phys. D}, 27(1):1209--1225, 1961.

\bibitem{Krieger70}
Wolfgang Krieger.
\newblock On entropy and generators of measure-preserving transformations.
\newblock {\em Trans. Amer. Math. Soc.}, 149:453--464, 1970.

\bibitem{Lieb}
Elliott~H. Lieb.
\newblock Residual entropy of square ice.
\newblock {\em Phys.Rev.}, 162, 1967.

\bibitem{LindPerron1984}
D.~A. Lind.
\newblock The entropies of topological {M}arkov shifts and a related class of
  algebraic integers.
\newblock {\em Ergodic Theory Dynam. Systems}, 4(2):283--300, 1984.

\bibitem{LindSurvey2004}
Douglas Lind.
\newblock Multi-dimensional symbolic dynamics.
\newblock In {\em Symbolic dynamics and its applications}, volume~60 of {\em
  Proc. Sympos. Appl. Math.}, pages 61--79. Amer. Math. Soc., Providence, RI,
  2004.

\bibitem{McgoffPavlov16}
Kevin McGoff and Ronnie Pavlov.
\newblock Random {$\Bbb{Z}^d$}-shifts of finite type.
\newblock {\em J. Mod. Dyn.}, 10:287--330, 2016.

\bibitem{Meyerovitch2010}
Tom Meyerovitch.
\newblock Growth-type invariants for {$\Bbb Z^d$} subshifts of finite type and
  arithmetical classes of real numbers.
\newblock {\em Invent. Math.}, 184(3):567--589, 2011.

\bibitem{Mozes}
Shahar Mozes.
\newblock Tilings, substitution systems and dynamical systems generated by
  them.
\newblock {\em J. Analyse Math.}, 53:139--186, 1989.

\bibitem{PavlovSchraudner2015}
Ronnie Pavlov and Michael Schraudner.
\newblock Entropies realizable by block gluing {$\Bbb{Z}^d$} shifts of finite
  type.
\newblock {\em J. Anal. Math.}, 126:113--174, 2015.

\bibitem{QuasSoo16}
Anthony Quas and Terry Soo.
\newblock Ergodic universality of some topological dynamical systems.
\newblock {\em Trans. Amer. Math. Soc.}, 368(6):4137--4170, 2016.

\bibitem{Robinson71}
Raphael~M. Robinson.
\newblock Undecidability and nonperiodicity for tilings of the plane.
\newblock {\em Invent. Math.}, 12:177--209, 1971.

\bibitem{Shelahweiss}
Saharon Shelah and Benjamin Weiss.
\newblock Measurable recurrence and quasi-invariant measures.
\newblock {\em Israel J. Math.}, 43(2):154--160, 1982.

\bibitem{Simpson2014}
Stephen~G. Simpson.
\newblock Medvedev degrees of two-dimensional subshifts of finite type.
\newblock {\em Ergodic Theory Dynam. Systems}, 34(2):679--688, 2014.

\bibitem{Tserunyan}
Anush Tserunyan.
\newblock Finite generators for countable group actions in the {B}orel and
  {B}aire category settings.
\newblock {\em Adv. Math.}, 269:585--646, 2015.

\bibitem{Weiss84}
Benjamin Weiss.
\newblock Measurable dynamics.
\newblock In {\em Conference in modern analysis and probability ({N}ew {H}aven,
  {C}onn., 1982)}, volume~26 of {\em Contemp. Math.}, pages 395--421. Amer.
  Math. Soc., Providence, RI, 1984.

\bibitem{Weiss89}
Benjamin Weiss.
\newblock Countable generators in dynamics---universal minimal models.
\newblock In {\em Measure and measurable dynamics ({R}ochester, {NY}, 1987)},
  volume~94 of {\em Contemp. Math.}, pages 321--326. Amer. Math. Soc.,
  Providence, RI, 1989.

\bibitem{Z2016}
Charalampos Zinoviadis.
\newblock Hierarchy and expansiveness in two-dimensional subshifts of finite
  type.
\newblock {\em Math arXiv}, 1603.05464, 2016.

\end{thebibliography}

\end{document}